\documentclass{article}
\usepackage[a4paper]{geometry}
\usepackage[T1]{fontenc}
\usepackage{ae}
\usepackage{amsfonts,amssymb}
\usepackage{url}
\usepackage{theorem}
\newtheorem{theorem}{Theorem}
\newtheorem{lemma}{Lemma}
{
\theoremstyle{break}
\theorembodyfont{\rmfamily}
\newtheorem{algorithm}{Algorithm}
}
\newenvironment{proof}{\textit {Proof. }}{\hspace* {\fill}$\square$}

\newcommand {\kronecker}[2]{\bigl( \frac{#1}{#2} \bigr)}
\newcommand {\llog}{\mathrm{llog}}
\newcommand {\lllog}{\mathrm{lllog}}
\newcommand{\ZZ}{\mathbf {Z} }      
\newcommand{\QQ}{\mathbf {Q} }		
\newcommand{\RR}{\mathbf {R} }      
\newcommand{\CC}{\mathbf {C} }      
\newcommand{\FF}{\mathbf {F} }		
\newcommand{\Fp}{\FF_p}			
\newcommand{\Fpt}{\FF_{p^2}}		
\newcommand{\Or}{\mathcal {O} }		
\newcommand{\Gal}{{\mathrm {Gal}} }
\newcommand{\Cl}{{\mathrm {Cl}} }
\newcommand{\End}{{\mathrm {End}} }
\newcommand{\Ell}{{\mathrm {Ell}} }
\newcommand{\Emb}{{\mathrm {Emb}} }

\newcommand{\disc}{{\rm {disc}} }
\newcommand {\Li}{\mathrm{Li}}

\newcommand{\ga}{\mathfrak{a}}		
\newcommand{\gb}{\mathfrak{b}}		
\newcommand{\gl}{\mathfrak{l}}		
\newcommand{\cP}{\mathcal {P} }		
\newcommand{\Ap}{\mathcal {A}_{p,\infty}} 
\newcommand{\isar}{\ \smash{\mathop{\longrightarrow}\limits^{\thicksim}}\ }

\begin{document}

\title{Computing Hilbert Class Polynomials}
\author{Juliana Belding$^1$,
Reinier Br\"oker$^2$,
Andreas Enge$^3$,
Kristin Lauter$^2$}
\date {\small $^1$ Dept.
of Mathematics, University of Maryland, College Park, MD 20742, USA, \\
\texttt {jbelding@math.umd.edu} \\
$^2$ Microsoft
Research, One Microsoft Way, Redmond, WA 98052, USA, \\
\texttt {reinierb@microsoft.com}, \texttt {klauter@microsoft.com} \\
$^3$ INRIA Futurs \& Laboratoire d'Informatique (CNRS/UMR 7161),
\'Ecole polytechnique, \\ 91128 Palaiseau cedex, France,
\texttt {enge@lix.polytechnique.fr}
}

\maketitle

\begin{abstract}
We present and analyze two algorithms for computing the Hilbert class
polynomial~$H_D$. The first is a $p$-adic lifting algorithm for
inert primes $p$ in the order of discriminant~$D<0$.
The second is an improved Chinese remainder algorithm which uses the class
group action on CM-curves over finite fields.
Our run time analysis gives tighter bounds for the complexity of all known
algorithms for computing~$H_D$, and we show that all methods have comparable
run times.
 
\end{abstract}

\section{Introduction} 	

For an imaginary quadratic order $\Or = \Or_D$ of discriminant $D<0$, the 
$j$-invariant of the complex elliptic curve $\CC/\Or$ is an algebraic
integer. Its minimal polynomial $H_D\in\ZZ[X]$ is called the {\it Hilbert
class polynomial\/}. It defines the ring class field $K_\Or$ corresponding
to $\Or$, and within the context of explicit class field theory, it is 
natural to ask for an algorithm to {\it explicitly compute\/}~$H_D$.

Algorithms to compute $H_D$ are also interesting
for elliptic curve primality proving \cite{AtMo93}
and for cryptographic purposes \cite{CoFrAvDoLaNgVe06}; for instance, 
pairing-based cryptosystems using ordinary curves rely on complex multiplication techniques to generate the curves.
The classical approach to compute $H_D$
is to approximate the values $j(\tau_\ga)\in\CC$ of the 
complex analytic $j$-function at points $\tau_\ga$ in the upper half plane 
corresponding to the ideal classes $\ga$ for the order $\Or$. The polynomial
$H_D$ may be recovered by rounding the coefficients of
$
\prod_{\ga \in \Cl(\Or)}\left( X- j(\tau_\ga) \right) \in \CC[X]
$
to the nearest integer. It is shown in \cite{En06} that an optimized version of that algorithm
has a complexity that is essentially linear in the output size.

Alternatively one can compute $H_D$ using a $p$-adic lifting algorithm 
\cite{CoHe02,Br06}. Here, the prime $p$ splits completely in $K_\Or$ and
is therefore relatively large: it satisfies the lower bound $p \geq
|D|/4.$
In this paper we give a $p$-adic algorithm for {\it inert\/}
primes~$p$.
Such primes are typically much smaller than totally split primes, and under
GRH there exists an inert prime of size only $O((\log |D|)^2)$.
The complex multiplication theory underlying all methods is more intricate
for inert primes $p$, as the roots of $H_D \in \Fpt[X]$ are now
$j$-invariants
of {\it supersingular\/} elliptic curves. In Section~\ref {sec:cm}
we explain how to
define the canonical lift of a supersingular elliptic curve,
and in Section~\ref {sec:lift} we describe a method to explicitly
compute this lift.

In another direction, it was suggested in \cite{AgLaVe04} to compute $H_D$ 
modulo several totally split primes $p$ and then combine the information 
modulo $p$ using the Chinese remainder theorem to compute $H_D \in
\ZZ[X]$. The first version of this algorithm was quite impractical, and in 
Section~\ref{sec:mp} we improve this `multi-prime approach' in two different 
ways. We show how 
to incorporate inert primes, and we improve the
original approach for totally split primes using the class group action
on CM-curves.
We analyze the run time of the new algorithm in Section~\ref{sec:analysis}
in terms of the logarithmic height of $H_D$, its degree, the largest prime
needed to generate the class group of $\Or$ and the 
discriminant~$D$. Our tight bounds on the first two quantities from
Lemmata~\ref{boundh(D)} and \ref{bound1/a} apply to all methods to
compute~$H_D$. For the multi-prime approach, we derive the following result.
\begin{theorem}
\label {mainresult}
The algorithm presented in Section~\ref {sec:mp} computes,
for a discriminant $D<0$, the Hilbert class polynomial $H_D$.
If GRH holds true, the algorithm has an expected run time
\[
O \left( |D| (\log |D|)^{7+o(1)} \right).
\]
Under heuristic assumptions, the complexity becomes
\[
O \left( |D| (\log |D|)^{3+o(1)} \right).
\]
\end{theorem}
\noindent
We conclude by giving examples of the presented algorithms in Section~\ref
{sec:examples}.

\section{Complex multiplication in characteristic $p$} 
\label {sec:cm}

Throughout this section, $D < -4$ is any discriminant, and we write
$\Or$ for the imaginary quadratic order of discriminant~$D$. Let
$E/K_\Or$ be an elliptic curve with endomorphism ring isomorphic to $\Or$.
As
$\Or$ has rank 2 as a $\ZZ$-algebra, there are {\it two\/} isomorphisms
$\varphi: \End(E) \isar \Or.$
We always assume we have chosen the {\it normalized\/}
isomorphism, i.e., for all $y\in \Or$ we have $\varphi(y)^* \omega = 
y \omega$ for all invariant differentials $\omega$. For ease of notation, we
write $E$ for such a `normalized elliptic curve,' the isomorphism
$\varphi$ being understood.

For a field $F$, let $\Ell_D(F)$ be the set of isomorphism classes of 
elliptic curves over $F$ with endomorphism ring $\Or$. The ideal group of 
$\Or$ acts on $\Ell_D(K_\Or)$ via
$$
j(E) \mapsto j(E)^\ga = j(E/E[\ga]),
$$
where $E[\ga]$ is the group of $\ga$-torsion points, i.e., the points that
are annihilated by all $\alpha \in \ga \subset \Or = \End(E)$. As principal
ideals act trivially, this action factors through the class group
$\Cl(\Or)$.
The $\Cl(\Or)$-action is transitive and free, and $\Ell_D(K_\Or)$ is a
principal homogeneous $\Cl(\Or)$-space.

Let $p$ be a prime that splits completely in the ring class field $K_\Or$.
We can embed $K_\Or$ in the $p$-adic field $\QQ_p$,
and the reduction 
map $\ZZ_p \rightarrow \Fp$ induces a bijection
$\Ell_D(\QQ_p) \rightarrow \Ell_D(\Fp)$. The $\Cl(\Or)$-action respects
reduction modulo~$p$, and the set $\Ell_D(\Fp)$ is a
$\Cl(\Or)$-torsor, just like in characteristic zero. This observation is
of key importance for the improved `multi-prime' approach explained in
Section~\ref {sec:mp}. 

We now consider a prime $p$ that is {\it inert\/} in $\Or$, fixed for the
remainder of this section. As the principal
prime $(p) \subset \Or$ splits completely in $K_\Or$, all primes of
$K_\Or$ lying over $p$ have residue class degree~$2$. We view $K_\Or$ as
a subfield of the unramified degree $2$ extension $L$ of $\QQ_p$.
It is a classical result, see~\cite{De41} or \cite[Th.~13.12]{La87}, that for $[E] \in \Ell_D(L)$, the
reduction $E_p$ is {\it supersingular\/}. It can be defined over the finite
field $\Fpt$, and its endomorphism ring is a maximal order in the unique
quaternion algebra $\Ap$ which is ramified at $p$ and~$\infty$.
The reduction map $\ZZ_L \rightarrow \Fpt$ also induces an embedding
$f: \Or \hookrightarrow \End(E_p)$. This embedding is not
surjective, as it is in the totally split case, since $\End(E_p)$ has
rank $4$ as a $\ZZ$-algebra, and $\Or$ has rank~2.

We let $\Emb_D(\Fpt)$ be the set of isomorphism classes of pairs
$(E_p ,f)$ with $E_p/\Fpt$ a supersingular elliptic curve and $f:
\Or \hookrightarrow \End( E_p)$ an embedding.
Here, $( E_p, f)$ and $( E_p', f')$ are isomorphic if there
exists an isomorphism $h: E_p \isar E_p'$ of elliptic
curves with $h^{-1} f'(\alpha) h = f(\alpha)$ for all $\alpha\in\Or$.
As an analogue of picking the normalized isomorphism $\Or \isar \End(E)$ in
characteristic zero, we now identify $( E_p, f)$ and $( E_p', f')$
if $f$ equals the complex conjugate of $f'$.

\begin{theorem}Let $D<-4$ be a discriminant. If $p$ is inert
in
$\Or = \Or_D$, the reduction map
$
\pi: \Ell_D(L) \rightarrow \Emb_D(\Fpt)
$
is a bijection. Here, $L$ is the unramified extension of $\QQ_p$ of degree
2.
\end{theorem}
\begin {proof}
By the Deuring lifting theorem, see \cite{De41} or \cite[Th.~13.14]{La87}, 
we can lift an element of $\Emb_D(\Fpt)$ to an element of $\Ell_D(L)$.
Hence, the map is surjective.

Suppose that we have $\pi(E) = \pi(E')$. As $E$ and $E'$ both have endomorphism
ring $\Or$, they are isogenous. We let 
$\varphi_\ga: E \rightarrow E^\ga = E'$ be an isogeny. Writing $\Or =
\ZZ[\tau]$, we get
$$
f' = f^\ga: \tau \mapsto \overline\varphi_\ga f(\tau) \widehat{\overline\varphi}_\ga 
\otimes (\deg \overline\varphi_\ga)^{-1} \in \End(E_p) \otimes \QQ.
$$
The map $\overline\varphi_\ga$ commutes with $f(\tau)$ and 
is thus contained in $S = f(\End(E)) \otimes \QQ$.

Write $\Or' = S \cap \End(E_p)$, and let $m$ be the index
$[\Or':f(\End(E))]$.
For any $\delta\in\Or'$, there exists $\gamma\in\End(E)$ with $m\delta =
f(\gamma)$. As $f(\gamma)$ annihilates the $m$-torsion $E_p[m]$,
 $\gamma$ annihilates $E[m]$, thus it is a multiple of $m$ inside $\End(E)$. We
derive that $\delta$ is contained in $f(\End(E))$, and $\Or' = f(\End(E))$.
Hence, $\varphi_\ga$ is an endomorphism of $E$, and $E$ and $E^\ga$ are
isomorphic.
\end {proof}\par
\ \par\noindent
The {\it canonical lift\/} $\widetilde E$ of a pair $(E_p,f) \in 
\Emb_D(\Fpt)$ is defined as the inverse $\pi^{-1}(E_p,f) \in \Ell_D(L)$.
This generalizes the notion
of a canonical lift for ordinary elliptic curves, and the main step of
the $p$-adic algorithm described in Section~\ref {sec:lift} is to compute
$\widetilde E$: its $j$-invariant is a zero of
the Hilbert class polynomial $H_D \in L[X]$.

The reduction map $\Ell_D(L) \rightarrow \Emb_D(\Fpt)$ induces a transitive
and free action
of the class group on the set $\Emb_D(\Fpt)$. For an $\Or$-ideal $\ga$,
let $\varphi_\ga: E 
\rightarrow E^\ga$ be the isogeny of CM-curves with kernel $E[\ga]$. Writing
$\Or = \ZZ[\tau]$, let $\beta\in\End(E)$ be the image of $\tau$ under the
normalized isomorphism $\Or \isar \End(E)$. The normalized isomorphism for
$E^\ga$ is now given by
$$
\tau \mapsto \varphi_\ga \beta \widehat\varphi_\ga
\otimes (\deg\varphi_\ga)^{-1}.
$$
We have $E_p^\ga = (E^\ga)_p$ and $f^\ga$ is the composition
$\Or \isar \End(E^\ga) 
\hookrightarrow \End(E_p^\ga)$. Note that principal ideals indeed act trivially:
$\varphi_\ga$ is an endomorphism in this case and, as $\End(E)$ is
commutative, we have $f=f^\ga$. 

To explicitly compute this action, we fix one supersingular curve
$E_p/\Fpt$ 
and an isomorphism $i_{E_p}: \Ap \isar \End(E_p) \otimes \QQ$
and view the embedding $f$ as an injective map $f: \Or 
\hookrightarrow \Ap$. Let $R = i_{E_p}^{-1}(\End( E_p))$ be the maximal
order
of $\Ap$ corresponding to $E_p$.
For $\ga$ an ideal of $\Or$, we  compute the curve $E_p^\ga = 
\overline\varphi_\ga(E_p)$ and choose an auxiliary isogeny  
$\varphi_\gb: E_p \rightarrow E_p^\ga$. This induces an isomorphism 
$g_\gb: \Ap \isar \End(E_p^\ga) \otimes \QQ$ given by 
$$
\alpha \mapsto \varphi_\gb i_{ E_p}(\alpha)\widehat{\varphi}_\gb \otimes 
(\deg \varphi_\gb)^{-1}. 
$$ 
The left $R$-ideals $Rf(\ga)$ and $\gb$ are left-isomorphic by 
\cite[Th.~3.11]{Wa69} and thus 
we can find $x \in \Ap$ with $R f(\ga) = \gb x$.
As $y=f(\tau)$ is an element of $Rf(\ga)$, we get the 
embedding $\tau \mapsto xyx^{-1}$ into
the right order $R_\gb$ of $\gb$. By construction, the induced embedding 
$f^\ga: \Or \hookrightarrow \End(E_p^\ga)$ is precisely
$$
f^\ga(\tau) = g_\gb(xyx^{-1}) \in \End(E_p^\ga), 
$$
and this is independent of the choice of $\gb$. For example, if $E_p^\ga =
E_p$, 
then choosing $\varphi_\gb$ as the identity, we find $x$ with 
$Rf(\ga) = Rx$ to get the embedding $f^\ga: \tau \mapsto i_{E_p}(xyx^{-1}) \in
\End(E_p)$.

\section{The multi-prime approach} 
\label {sec:mp}

This section is devoted to a precise description of the new algorithm for
computing the Hilbert class polynomial $H_D\in\ZZ[X]$ via the Chinese
remainder theorem.

\begin {algorithm}
\label {alg:crt}
\textsc {Input}: an imaginary quadratic discriminant $D$ \\
\textsc {Output}: the Hilbert class polynomial $H_D\in\ZZ[X]$

\begin {enumerate}
\setcounter {enumi}{-1}
\item
Let $(A_i,B_i,C_i)_{i=1}^{h (D)}$ be the set of primitive
reduced binary quadratic forms of discriminant $B_i^2-4A_iC_i=D$
representing the class group $\Cl(\Or)$. Compute
\begin {equation}
\label {eq:n}
n = \left\lceil \log_2 \left( 2.48 \, h(D) + \pi \sqrt {|D|}
\sum_{i=1}^{h (D)}
\frac {1}{A_i} \right) \right\rceil + 1,
\end {equation}
which by \cite{En06} is an upper bound on the number of bits in the largest
coefficient of $H_D$.
\item
Choose a set $\cP$ of primes $p$ such that $N = \prod_{p\in\cP} p \geq 2^n$
and each $p$ is either inert in $\Or$ or totally split in $K_\Or$. 
\item
For all $p \in \cP$, depending on whether $p$ is split or inert in $\Or$,
compute
$H_D \bmod p$ using either Algorithm~\ref {alg:split} or~\ref {alg:inert}.
\item
Compute $H_D \bmod N$ by the Chinese remainder theorem, and return its 
representative in $\ZZ[X]$ with coefficients in 
$\left( - \frac {N}{2}, \frac {N}{2} \right)$.
\end {enumerate}
\end {algorithm}

The choice of $\cP$ in Step~1 leaves some room for different flavors of the
algorithm. Since Step~2 is exponential in $\log p$, the primes should be
chosen as small as possible. The simplest case is to only use split primes,
to be analyzed in Section~\ref {sec:analysis}.
As the run time of Step 2 is worse for inert primes than for
split primes, we view the use of inert primes as a practical
improvement.

\subsection{Split primes}

A prime $p$ splits completely in $K_\Or$ if and only if the equation
$
4p = u^2-v^2D
$
has a solution
in integers $u,v$. For any prime $p$, we can efficiently test if such a solution
exists using an algorithm due to Cornacchia.
In practice, we generate primes satisfying 
this relation by varying $u$ and $v$ and testing if $(u^2-v^2D)/4$ is prime.

\begin {algorithm}
\label {alg:split}
\textsc {Input}: an imaginary quadratic discriminant $D$ and a prime $p$ that 
splits completely in $K_\Or$ \\
\textsc {Output}: $H_D \bmod p$

\begin {enumerate}
\item
Find a curve $E$ over $\FF_{p}$ with endomorphism ring $\Or$. Set
$j = j(E)$.
\item
Compute the Galois conjugates $j^\ga$ for $\ga \in \Cl(\Or)$.
\item
Return $H_D \bmod p = \prod_{\ga \in \Cl(\Or)} (X - j^\ga)$.
\end {enumerate}
\end {algorithm}

\noindent
{\bf Note:}
The main difference between this algorithm and the one proposed in \cite
{AgLaVe04} is that the latter determines \textit
{all} curves with endomorphism ring $\Or$ via exhaustive search, while we search for
one and obtain the others via the action of $\Cl(\Or)$ on the set $\Ell_D(\Fp)$.

Step~1
can be implemented by picking $j$-invariants at random until one with the 
desired endomorphism ring is found. With $4p=u^2-v^2D$, a necessary condition 
is that the curve $E$ or its quadratic twist $E'$ has $p+1-u$ points.
In the case that $D$ is fundamental and $v=1$, this condition is also 
sufficient. To test if one of our curves $E$ has the right cardinality, we pick 
a random point $P\in E(\Fp)$ and check if $(p+1-u)P = 0$ or $(p+1+u)P = 0$ 
holds. If neither of them does, $E$ does not have endomorphism ring $\Or$. 
If $E$ survives this test, we select a few random points on both $E$ and 
$E'$ and compute the orders of these points {\it assuming\/} they 
divide $p+1\pm u$. 
If the curve $E$ indeed has $p+1\pm u$ points, we quickly find 
points $P \in E(\Fp)$, $P' \in E'(\Fp)$ of
maximal order, since we have $E(\Fp) \cong \ZZ/n_1\ZZ \times \ZZ/n_2\ZZ$ 
with $n_1 \mid n_2$ and a fraction $\varphi(n_2)/n_2$ of the points have 
maximal order. For $P$ and $P'$ of maximal order and $p>457$,
either the order of $P$ or the order of $P'$
is at least $4\sqrt{p}$, by \cite[Theorem 3.1]{Sc95}, due to J.-F. Mestre. 
As the Hasse interval has length~$4 \sqrt{p}$, this
then {\it proves\/} that $E$ has $p+1\pm u$ points.

Let $\Delta = \frac {D}{f^2}$ be the fundamental discriminant associated to
$D$. For $f \neq 1$ or $v \neq 1$ (which happens necessarily for $D \equiv 1
\bmod 8$), the curves with $p+1\pm u$ points admit any order $\Or_{g^2 \Delta}$
such that $g | f v$ as their endomorphism rings. In this case, one possible
strategy is to use Kohel's algorithm described in \cite[Th.~24]{Ko96} to
compute $g$, until a curve with $g = f$ is found. This variant is easiest
to analyze and enough to prove Theorem~\ref {mainresult}.

In practice, one would rather keep a curve that satisifes $f | g$, since
by the class number formula $g = vf$ with overwhelming probability.
As $v$ and thus $\frac {fv}{g}$ is small, it is then
possible to use another algorithm due to Kohel and analyzed in detail by
Fouquet--Morain \cite{Ko96,FoMo02} to quickly apply an isogeny of degree
$\frac {fv}{g}$ leading to a curve with endomorphism ring $\Or$.

Concerning Step~2, let $\Cl(\Or) = \bigoplus \langle \gl_i \rangle$ be a
decomposition of the class group into a direct product of cyclic groups
generated by invertible degree~$1$ prime ideals $\gl_i$ of order $h_i$
and
norm $\ell_i$ not dividing $pv$. The $j^\ga$ may then be obtained
successively by computing the Galois action of the $\gl_i$ on
$j$-invariants
of curves with endomorphism ring $\Or$ over $\Fp$, otherwise said, by
computing $\ell_i$-isogenous curves: $h_1 - 1$ successive applications of
$\gl_1$ yield $j^{\gl_1}$, \ldots, $j^{\gl_1^{h_1 - 1}}$; to each of them,
$\gl_2$ is applied $h_2 - 1$ times, and so forth.

To explicitly compute the action of $\gl = \gl_i$, we let $\Phi_\ell(X,Y)
\in \ZZ[X]$ be the classical modular polynomial. It is a model for the
modular curve $Y_0(\ell)$ parametrizing elliptic curves together with an
$\ell$-isogeny, and it satisfies $\Phi_\ell (j (z), j (\ell z)) = 0$ for the
modular function $j (z)$. If $j_0 \in \FF_{p}$ is the $j$-invariant of
some curve with endomorphism ring $\Or$, then all the roots in $\Fp$ of
$\Phi_\ell(X, j_0)$ are $j$-invariants of curves with endomorphism ring
$\Or$ by \cite[Prop.~23]{Ko96}. If $\gl$ is unramified, there are two roots,
$j_0^{\gl}$ and $j_0^{\gl^{-1}}$. For ramified $\gl$, we find only one root
$j_0^{\gl} = j_0^{\gl^{-1}}$. So Step~2 is reduced to determining roots of
univariate polynomials over $\FF_{p}$.

\subsection{Inert primes}
\label {ssec:ex inert}

\begin {algorithm}
\label {alg:inert}
\textsc {Input}: an imaginary quadratic discriminant $D$ and
a prime $p$ that is inert in $\Or$ \\
\textsc {Output}: $H_D \bmod p$

\begin {enumerate}
\item Compute the list of supersingular $j$-invariants over $\Fpt$ together
with
their endomorphism rings inside the quaternion algebra $\Ap$.
\item Compute an optimal embedding $f: \Or \hookrightarrow \Ap$ and let $R$
be
a maximal order that contains $f(\Or)$.
\item Select a curve $E/\Fpt$ in the list with $\End(E) \cong R$, and let
$j$
be its $j$-invariant.
\item Compute the Galois conjugates $j^\ga$ for $\ga \in \Cl(\Or)$.
\item Return $H_D \bmod p = \prod_{\ga \in \Cl(\Or)} (X - j^\ga)$.
\end {enumerate}
\end {algorithm}
As the number of supersingular $j$-invariants grows roughly like 
$(p-1)/12$, this algorithm is only feasible for {\it small\/} primes. 
For the explicit computation, we 
use an algorithm due to Cervi\~no \cite{Ce04} to compile our list. The
list gives a bijection between the set of $\Gal(\Fpt/\Fp)$-conjugacy
classes of supersingular $j$-invariants and the set of maximal orders
in~$\Ap$.

In Step~2 we compute an element $y\in\Ap$ satisfying the same minimal
polynomial as a generator $\tau$ of $\Or$. For non-fundamental discriminants
we need to ensure that the embedding is optimal, i.e., does not extend to an
embedding of the maximal overorder of $\Or$ into $\Ap$. Using 
standard algorithms for quaternion algebras, Step~2  poses no practical
problems.
To compute the action of an ideal $\ga$ in Step 4, we note that 
the right order $R'$ of the left $R$-ideal $Rf(\ga)$ is isomorphic to the
endomorphism ring $\End(E')$ of a curve $E'$ with $j(E') = j^\ga$ by
\cite[Prop.~3.9]{Wa69}. The order $R'$ is isomorphic to a unique order in 
the list, and we get a conjugacy class of supersingular $j$-invariants. Since roots of
$H_D \bmod p$ which are not in $\Fp$ come in conjugate pairs, this allows us to compute all the
Galois conjugates $j^\ga$.

\section{Computing the canonical lift of a supersingular curve} 
\label{sec:lift}

In this section we explain how to compute the Hilbert class polynomial
$H_D$ of a discriminant $D<-4$ using a $p$-adic lifting technique for 
an inert prime $p \equiv 1 \bmod 12$. Our approach is based on the outline
described in \cite{CoHe02}. The condition $p \equiv 1 \bmod 12$ 
ensures that the $j$-values $0,1728 \in \Fp$ are not roots of 
$H_D \in \Fp[X]$. The case where one of these two values is a root of
$H_D\in\Fp[X]$ is more technical due to the extra automorphisms of the
curve, and will be explained in detail in the first author's PhD thesis.

Under GRH, we can take $p$ to be {\it small\/}. Indeed,
our condition amounts to prescribing a Frobenius symbol in the degree 8
extension $\QQ(\zeta_{12},\sqrt{D})/\QQ$, and by effective Chebotarev
\cite{LaOd77} we may take $p$ to be of size $O((\log |D|)^2)$.

The first step of the algorithm is the same as for
Algorithm~\ref {alg:inert} in Section~\ref {sec:mp}:
we compute a pair $( j(E_p) , f_0) \in \Emb_D(\Fpt)$. 
The main step of the algorithm is to compute to sufficient $p$-adic
precision the canonical lift $\tilde E_p$ of this pair, defined in 
Section~\ref {sec:cm} as the inverse under the bijection $\pi$ of Theorem~2.

For an arbitrary element $\eta \in \Emb_D(\Fpt)$, let
$$
X_D(\eta) = \{ (j(E),f) \mid j(E) \in \CC_p, (j(E) \bmod p ,f) = \eta \}
$$
be a `disc' of pairs lying over $\eta$. Here, $\CC_p$ is the completion
of an algebraic closure of $\QQ_p$. The disc $X_D(\eta)$ contains the points
of $\Ell_D(L)$ that reduce modulo $p$ to the $j$-invariant corresponding to~$\eta$.

These discs are similar to the discs used for the {\it split\/} case in
\cite{CoHe02,Br06}. The main difference is that now we need to keep track of the
embedding as well. We can adapt the key idea of \cite{CoHe02} to
construct a $p$-adic analytic map from the set of discs to itself that has
the CM-points as fixed points in the following way. Let $\ga$ be an $\Or$-ideal 
of norm $N$ that is coprime to $p$. We define a map
$$
\rho_\ga: \bigcup_\eta X_D(\eta) \rightarrow \bigcup_\eta X_D(\eta)
$$
as follows. For $(j(E),f) \in X_D(\eta)$, the ideal $f(\ga) \subset 
\End(E_p)$ defines a subgroup $E_p[f(\ga)] \subset 
 E_p[N]$ which lifts canonically to a subgroup $E[\ga] \subset E[N]$. We
define $\rho_\ga((j(E),f)) = (j(E/E[\ga]),f^\ga)$, where $f^\ga$ is as in
Section~\ref {sec:cm}.
If the map $f$ is clear, we also denote by $\rho_\ga$ the induced map
on the $j$-invariants.

For principal ideals $\ga = (\alpha)$, the map $\rho_\ga = \rho_\alpha$
stabilizes
every disc. Furthermore, as $\widetilde E_p[(\alpha)]$ determines an endomorphism
of
$\widetilde E_p$, the map $\rho_\alpha$ {\it fixes\/} the canonical lift
$j(\widetilde E_p)$. As $j(E_p)$ does not equal $0,1728\in\Fp$, the
map $\rho_\alpha$ is {\it $p$-adic analytic\/} by \cite[Theorem~4.2]{Br06}.

Writing $\alpha = a + b\tau$, the derivative of $\rho_\alpha$ in a
CM-point $j(\widetilde E)$ equals $\alpha/\overline\alpha\in \ZZ_L$ by
\cite[Lemma~4.3]{Br06}. For $p \nmid a,b$ this is a $p$-adic unit and we
can use a modified version of Newton's method to converge to
$j(\widetilde E)$ starting from a random lift $(j_1,f_0)\in X_D(\eta)$
of the chosen point $\eta = (j(E_p),f_0)\in\Fpt$. Indeed, the sequence
\begin {equation}
\label {eq:newton}
j_{k+1} = j_k - \frac{\rho_\alpha((j_k,f_0)) - j_k}
{\alpha/ \overline{\alpha} - 1} 
\end {equation}
converges quadratically to $j(\widetilde E$).
The run time of the resulting algorithm to compute $j(\widetilde E) \in L$
up to the necessary precision depends heavily on the choice of~$\alpha$. We find a suitable $\alpha$ by sieving in the
set $\{ a + b \tau \mid a,b \in \ZZ, \gcd(a,b) = 1, a, b \neq 0 \bmod p \}.$
We refer to the example in Section 6.3 for the explicit computation of the
map $\rho_\alpha$.

Once the canonical lift has been computed, the computation
of the Galois conjugates is easier. 
To compute the Galois conjugate $j(\widetilde E_p)^\gl$
of an ideal $\gl$ of prime norm $\ell \neq p$, we first compute the value
$j(E_p)^\gl\in\Fpt$ as in Algorithm~\ref {alg:inert} in Section~\ref {sec:mp}.
We then compute all roots of
the $\ell$-th modular polynomial $\Phi_\ell(j(\widetilde E_p),X) \in L[X]$ that
reduce to $j(E_p)^\gl$. If there is only one such root, we are done: this is
the Galois conjugate we are after. In general, if $m \geq 1$ is the $p$-adic
precision required to distinguish the roots, we compute the value
$\rho_\gl((j(\widetilde E_p),f_0))$ to $m+1$ $p$-adic digits precision to 
decide which root of the modular polynomial is the Galois conjugate.
After computing all conjugates, we expand the product
$
\prod_{\ga \in \Cl (\Or)}\left(X - j(\widetilde E_p)^\ga\right) 
\in \ZZ_L[X]
$
and recognize the coefficients as integers.

\section{Complexity analysis} 
\label {sec:analysis}

This section is devoted to the run time analysis of Algorithm~\ref
{alg:crt} and the proof of Theorem~\ref {mainresult}. 
To allow for an easier comparison with other methods to compute~$H_D$,
the analysis is
carried out with respect to all relevant variables: the discriminant~$D$,
the class number $h (D)$, the logarithmic height $n$ of the class
polynomial and the largest prime generator $\ell (D)$ of the class group,
before deriving a coarser bound depending only on $D$.

\subsection {Some number theoretic bounds}
\label {ssec:number theory bounds}

For the sake of brevity, we write $\llog$ for $\log \log$ and $\lllog$ for
$\log \log \log$.

The bound given in Algorithm~\ref {alg:crt} on $n$, the bit size of
the largest coefficient
of the class polynomial, depends essentially on two quantities: the class
number $h (D)$ of $\Or$ and the sum $\sum_{[A,B,C]} \frac {1}{A}$, taken
over a system of primitive reduced quadratic forms representing the class 
group $\Cl(\Or)$.

\begin{lemma}
\label{boundh(D)}
We have $h(D) = O(|D|^{1/2}\log |D|)$.
Under GRH, we have $h(D) = O (|D|^{1/2}\, \llog\, |D|)$.
\end{lemma}
\begin {proof}
By the analytic class number formula, we have to bound the value of
the Dirichlet $L$-series $L(s,\chi_D)$ associated to $D$ at $s=1$. The
unconditional bound follows directly from \cite{Sc18}, the conditional bound
follows from \cite{Li28}.
\end {proof}

\begin{lemma}
\label{bound1/a}
We have $\sum_{[A,B,C]} \frac{1}{A} = O((\log |D|)^2)$.
If GRH holds true, we have $\sum_{[A,B,C]} \frac{1}{A} = O (\log |D|\, \llog\, |D|)$.
\end{lemma}
\begin {proof}
The bound $\sum_{[A,B,C]} \frac{1}{A} = O((\log |D|)^2)$ is proved 
in \cite{Sc91} with precise constants in \cite{En06}; the argument below
will give a different proof of this fact.

By counting the solutions of $B^2 \equiv D \bmod 4A$ for varying $A$ and
using the Chinese remainder theorem, we obtain
$$
\sum_{[A,B,C]} \frac{1}{A}
\leq
\sum_{A \leq \sqrt{|D|}} \frac{\prod_{p \mid A}\left(1 +\kronecker {D}{p}
\right)}{A}.
$$
The Euler product expansion bounds this by $ \prod_{p\leq\sqrt{|D|}}
\left(1+\frac{1}{p}\right)\left(1+\frac {\kronecker{D}{p}}{p}\right)$. By
Mertens theorem, this is at most $c\log|D| \prod_{p \leq \sqrt{|D|}}\frac
{1}{1 - \kronecker{D}{p}/p}$ for some constant $c>0$. This last product is 
essentially the value of the Dirichlet $L$-series $L(1,\chi_D)$ and the
same remarks as in Lemma~\ref{boundh(D)} apply.
\end {proof}

\begin{lemma}
\label {boundprime}
If GRH holds true, the primes needed for Algorithm~\ref{alg:crt} 
are bounded by \\
$O \left( h (D) \max (h(D) (\log |D|)^4, n) \right)$.
\end{lemma}

\begin {proof}
Let $k(D)$ be the required number of splitting primes. We have 
$k(D) \in O \left( \frac {n}{\log |D|} \right)$, since
each prime has at least $\log_2 |D|$ bits.

Let $\pi_1(x,K_\Or/\QQ)$ be the number of primes up to $x \in \RR_{>0}$ that
split completely in $K_\Or/\QQ$. By \cite[Th.~1.1]{LaOd77} there
is an effectively computable constant $c \in \RR_{>0}$, independent of
$D$, such that
\begin {equation}
\label {eq:pi}
\left| \pi_1(x,K_\Or/\QQ) - \frac{\Li (x)}{2h(D)} \right| \leq c\left( \frac
{x^{1/2}
\log(|D|^{h(D)}x^{2h(D)})}{2h(D)} + \log(|D|^{h(D)}) \right),
\end {equation}
where we have used the bound $\disc(K_\Or/\QQ) \leq |D|^{h(D)}$ proven in 
\cite[Lemma 3.1]{Br06}.
It suffices to find an $x \in \RR_{>0}$ for which 
$k(D) - \Li(x)/(2h(D))$ is larger than the right hand side of (\ref
{eq:pi}).
Using the estimate $\Li(x) \sim x/\log x$, we see that 
the choice $x = O \left( \max (h(D)^2 \log^4 |D|, h (D) n) \right)$ works.
\end {proof}

\subsection {Complexity of Algorithm~\ref {alg:split}}
\label {ssec:split complexity}

Let us fix some notation and briefly recall the complexities of the
asymptotically fastest algorithms for basic arithmetic.
Let $M (\log p) \in O (\log p\, \llog\, p\, \lllog\, p)$ be the time for
a multiplication in $\FF_p$ and $M_X (\ell, \log p) \in O (\ell \log
\ell \, M (\log p))$ the time for multiplying two polynomials over $\FF_p$
of degree $\ell$.

As the final complexity will be exponential in $\log p$, we need not worry 
about the detailed complexity of polynomial or subexponential steps. 
Writing $4 p = u^2 - v^2 D$ takes polynomial time by the Cornacchia and 
Tonelli--Shanks algorithms \cite[Sec~1.5]{Co96}. By Lemma~\ref {boundprime}, 
we may assume that $v$ is polynomial in $\log |D|$.

Concerning Step~2, we expect to check $O (p / h (D))$ curves until finding
one with endomorphism ring $\Or$. To test if a curve has the desired 
cardinality, we need to compute the orders of $O(\llog\, p)$ points, and
each order computation takes time $O \left( (\log p)^2\, M(\log p) \right)$.
Among the curves with the right
cardinality, a fraction of $\frac{h(D)}{H(v^2 D)}$, where $H (v^2 D)$ is
the Kronecker class number, has the desired endomorphism ring. So we expect
to apply Kohel's algorithm with run time $O(p^{1/3+o(1)})$ an expected
$\frac{H(v^2 D)}{h(D)} \in O (v\, \llog\, v)$ times. As $p^{1/3}$ is dominated
by
$p / h (D)$ of order about $p^{1/2}$, Step~2 takes time altogether
\begin {equation}
\label {eq:point counting}
O \left( \frac {p}{h (D)} (\log p)^2 \, M (\log p)\, \llog\, p \right).
\end {equation}
Heuristically, we only check if some random points are annihilated
by $p+1\pm u$ and do not compute their actual orders.
The $(\log p)^2$ in (\ref{eq:point counting}) then becomes $\log p$.

In Step~3, the decomposition of the class group into a product
of cyclic groups takes subexponential time. Furthermore, since all involved 
primes $\ell_i$ are of size $O ((\log |D|)^2)$ under GRH, the time needed
to compute the modular polynomials is negligible. Step~3 is thus dominated 
by $O (h (D))$ evaluations of reduced modular polynomials and by the 
computation of their roots.

Once $\Phi_\ell \bmod p$ is computed, it can be evaluated in time 
$O (\ell^2 M (\log p))$. Finding its roots is dominated by
the computation of $X^p$ modulo the specialized polynomial of degree
$\ell + 1$, which takes time $O (\log p \, M_X (\ell, \log p))$.
Letting $\ell (D)$ denote the largest prime needed to generate the class
group, Step~3 takes time
\begin {equation}
\label {eq:conjugates}
O \left(
h (D) \ell (D) \, M (\log p) (\ell (D) + \llog\, |D| \log p)
\right).
\end {equation}
Under GRH, $\ell (D) \in O ((\log |D|)^2)$, and heuristically,
$\ell (D) \in O\left( (\log |D|)^{1 + \varepsilon} \right)$.

By organizing the multiplications of polynomials in a tree of height $O
(\log h)$, Step~4 takes $O (\log h (D) \, M_X (h (D), \log p))$, which
is dominated by Step~3. We conclude that the total complexity of 
Algorithm~\ref {alg:split} is dominated 
by Steps~2 and~3 and given by the sum of (\ref{eq:point counting}) 
and (\ref{eq:conjugates}).

\subsection {Proof of Theorem~\ref {mainresult}}

We assume that $\cP = \{ p_1, p_2, \ldots \}$ is chosen as the set of the 
smallest primes $p$ that split into principal ideals of $\Or$.
Notice that $\log p$, $\log h (D) \in O (\log |D|)$, so that we
may express all logarithmic quantities with respect to $D$.

The dominant part of the algorithm are the $O(n/\log |D|)$ invocations of 
Algorithm~\ref {alg:split} in Step~2. 
Specializing (\ref {eq:point counting}) and (\ref {eq:conjugates}),
using the bound on the largest prime of Lemma~\ref {boundprime}
and assuming that $\ell (D) \in \Omega (\log |D| \llog\, |D|)$,
this takes time
\begin {equation}
\label {eq:complexity}
O \left( \hspace* {-0.5mm} n  \, M (\log |D|)
\left( h (D) \frac {\ell (D)^2}{\log |D|} + \log |D| \llog |D|
\max \left( h (D) (\log |D|)^4, n \right)
\right) \hspace* {-0.5mm} \right) \hspace* {-0.5mm}.
\end {equation}

Finally, the fast Chinese remainder algorithm takes $O (M (\log
N) \llog\, N)$ by \cite[Th.~10.25]{GaGe99}, so that Step~3 can be carried out
in $O (h (D) \, M (n) \log |D|)$, which is also dominated by Step~2.
Plugging the bounds of Lemmata~\ref {boundh(D)} and~\ref {bound1/a} into
(\ref {eq:complexity}) proves the rigorous part of
Theorem~\ref {mainresult}.

For the heuristic result, we note that Lemma~\ref {boundprime}
overestimates the size of the primes, since it
gives a very high bound already for the \textit {first} split prime.
Heuristically, one would rather expect that all primes are of
size $O (n h)$.
Combined with the heuristic improvements to
(\ref{eq:point counting}) and (\ref{eq:conjugates}), we find the run time\par
\medskip
\hfill
$
\displaystyle{ O \left( n \, M (\log |D|)
\left (n + h (D) \frac {\ell (D)^2}{\log |D|} \right) \right).} \hfill\square
$ 

\subsection {Comparison}
\label {ssec:comparison}

The bounds under GRH of Lemmata~\ref {boundh(D)} and~\ref {bound1/a} also
yield a tighter analysis for other algorithms computing $H_D$.
By \cite[Th.~1]{En06}, the run time of the complex analytic algorithm
turns out to be
$
O (|D| (\log |D|)^3  (\llog\, |D|)^3),
$
which is essentially the same as the heuristic bound of
Theorem~\ref {mainresult}.

The run time of the $p$-adic algorithm becomes $O(|D| (\log |D|)^{6+o(1)})$.
A heuristic run time analysis of this algorithm has not been undertaken, but
it seems likely that $O(|D| (\log |D|)^{3+o(1)})$ would be reached again.

\section{Examples and practical considerations} 
\label {sec:examples}

\subsection{Inert primes}

For very small primes there is a unique supersingular $j$-invariant in 
characteristic~$p$. For example, for $D \equiv 5 \bmod 8$, the prime $p=2$
is inert in $\Or_D$ and we immediately have $H_D \bmod 2 = X^{h(D)}$.

More work needs to be done if there is more than one supersingular
$j$-invariant
in $\Fpt$, as illustrated by computing $H_{-71} \bmod 53$.
The ideal $\ga = (2, 3 + \tau)$ generates the order 7 class group of $
\Or = \ZZ[\tau]$. The quaternion algebra $\Ap$ has a
basis $\{ 1 , i , j , k \}$ with $i^2 = -2, j^2 = -35,  ij = k$, and the 
maximal order $R$ with basis $\{ 1, i, 1/4(2 - i - k), -1/2(1 + i +
j)\}$ 
is isomorphic to the endomorphism ring of the curve with $j$-invariant 50. We
compute 
the embedding $f: \tau \mapsto y = 1/2 - 3/2i + 1/2j \in R$, where $y$
satisfies
$y^2 - y + 18= 0$. Calculating the right
orders of the left $R$-ideals $Rf(\ga^i)$ for $i = 1,\ldots,7$, we get a
sequence of
orders corresponding to the $j$-invariants $28 \pm 9\sqrt{2},
46, 0, 46, 28 \pm 9\sqrt{2}, 50, 50$ and compute
$H_{-71} \bmod 53 = X(X-46)^2(X-50)^2(X^2 + 50X + 39).$

\subsection{Totally split primes}
For $D=-71$, the smallest totally split prime is
$p=107=\frac {12^2+4 \cdot 71}{4}$.
Any curve over $\Fp$ with endomorphism ring $\Or$ is
isomorphic to a curve with $m = p+1\pm 12 = 96$ or $120$ points. By trying
randomly chosen $j$-invariants, we find that $E: Y^2 = X^3+X+35$ has $96$
points.
We either have $\End(E) = \Or_D$ or $\End(E) = \Or_{4D}$.
In this simple case there is no need to apply Kohel's algorithm. Indeed,
$\End(E)$ equals $\Or_D$ if and only if the complete $2$-torsion is
$\Fp$-rational. The curve $E$ has only the point $P = (18,0)$ as rational
$2$-torsion point, and therefore has endomorphism ring $\Or_{4D}$. The
$2$-isogenous curve $E' = E/\langle P \rangle$ given by
$
Y^2 = X^3+58X+59
$
of $j$-invariant $19$ has endomorphism ring $\Or_D$.

The smallest odd prime generating the class group is $\ell=3$. The third
modular polynomial $\Phi_\ell(X,Y)$ has the two roots $46,63$ when evaluated
in $X = j(E') = 19 \in\Fp$. Both values are roots of $H_D \bmod p$. We
successively
find the other Galois conjugates $64$, $77$, $30$, $57$ using the modular
polynomial $\Phi_\ell$ and expand
$$
H_{-71}\bmod 107 = X^7 + 72X^6 + 93X^5 + 73X^4 + 46X^3 + 29X^2 + 30X + 19.
$$

\subsection{Inert lifting}
We illustrate the algorithm of Section~\ref {sec:lift} by computing
$H_D$ for $D=-56$.

The prime $p=37$ is inert in $\Or = \Or_D$. The supersingular $j$-invariants
in characteristic $p$ are $8,3\pm 14\sqrt{-2}$. We fix a
curve $E=E_p$ with  $j$-invariant~8. We take the basis $\{1,i,j,k\}$ with
$i^2=-2,j^2=j-5,ij=k$ of the quaternion algebra $\Ap$. This basis is also
a $\ZZ$-basis for a maximal order $R\subset\Ap$ that is isomorphic to the
endomorphism ring $\End(E_p)$.

Writing $\Or_D = \ZZ[\tau]$, we compute an element $y = [0,1,1,-1]\in R$
satisfying $y^2 + 56=0$. This determines the embedding $f=f_0$ and we need to
lift the pair $(E,f)$ to its canonical lift. As element $\alpha$ for the
`Newton map' $\rho_\alpha$, we use a generator of $\ga^4$ 
where $\ga=(3,1+\tau)$ is a prime lying over $3$.

To find the kernel $E[f(\ga)]$ we check which $3$-torsion points $P \in E[3]$
are killed by $f(1+\tau)\in\End(E)$. We find $P = 18 \pm 9\sqrt{-2}$, and 
use V\'elu's formulas to find $E^\ga \cong E$ of $j$-invariant 8. As $E$ and 
$E^\ga$ are isomorphic, it is easy to compute $f^\ga$. We compute a left-generator 
$x = [1,1,0,0] \in R$ of the left $R$-ideal $Rf(\ga)$ to find $f^\ga(\tau) = xy/x =
[-1,0,1,1] \in R$.

Next, we compute the $\ga$-action on the pair $(E^\ga,f^\ga)=(E,f^\ga)$. We find
that $P = 19 \pm 12\sqrt{a}$ is annihilated by $f^\ga(1+\tau)\in\End(E)$.
The curve
$E^{\ga^2}$ of $j$-invariant $3-14\sqrt{-2}$ is not isomorphic to $E$.
We pick a 2-isogeny $\varphi_\gb: 
E^\ga \rightarrow E^{\ga^2}$ with kernel $\langle 19+23\sqrt{-2}\rangle$.
The ideal $\gb$ has basis $\{2,i+j,2j,k\}$ and is left-isomorphic to $Rf^\ga(\ga)$ via
left-multiplication by
$x' = [-1,1/2,1/2,-1/2] \in R$. We get $f^{\ga^2}(\tau) = x'y/x' = [0,1,1,-1]
\in R_\gb$ and we use the map $g_\gb$ from Section~\ref {sec:cm} to view this as an embedding
into $\End(E^{\ga^2})$.

The action of $\ga^3$ and $\ga^4$ is computed in the same way. We find a cycle
of 3-isogenies
\[ 
(E, f) \rightarrow (E^\ga=E, f^\ga) \rightarrow (E^{\ga^2}, f^{\ga^2}) 
\rightarrow (E^{\ga^3}, f^{\ga^3}) \rightarrow (E^{\ga^4}, f^{\ga^4} ) = (E,f)
\]
where each element of the cycle corresponds
uniquely to a root of $H_D$. We have now also computed $H_D \bmod p = 
(X - 8)^2(X^2 - 6X - 6).$

As a lift of $E$ we choose the curve defined by $Y^2 = X^3+210X+420$ over
the unramified extension $L$ of degree 2 of $\QQ_p$. We lift the cycle
of isogenies over $\Fpt$ to $L$ in 2 $p$-adic digits precision using
Hensel's
lemma, and update according to the Newton formula (\ref {eq:newton}) to find
$j(\widetilde E) = -66+148\sqrt{-2} + O(p^2)$. Next we work with 4 $p$-adic
digits precision, lift the cycle of isogenies and update the $j$-invariant as
before. In this example, it suffices to work with $16$ $p$-adic digits
precision to recover $H_D \in \ZZ[X]$.

Since we used a generator of an ideal generating the class group, we get
the Galois conjugates of $j(\widetilde E)$ as a byproduct of our computation.
In the end we expand the polynomial
$
H_{-56} = \prod_{\ga \in \Cl(\Or)} (X - j(\widetilde E)^\ga) \in \ZZ[X]
$
which has coefficients with up to 23 decimal digits.

\subsection{Chinese remainder theorem}
\label {ssec:ex crt}

As remarked in Section~\ref {ssec:comparison}, the heuristic run time of
Theorem~\ref {mainresult} is 
comparable to the expected run times of both the complex analytic and the
$p$-adic approaches from \cite{En06} and \cite{CoHe02,Br06}. To see
if the CRT-approach is comparable {\it in practice\/} as well, we computed
an example with a reasonably sized discriminant $D=-108708$, the first
discriminant with class number~$100$.

The \textit {a posteriori} height of $H_D$ is $5874$ bits,
and we fix a target precision of $n=5943$. The smallest totally split prime
is $27241$. If only such primes are used, the largest one is $956929$
for a total of $324$ primes. Note that these primes are indeed of size 
roughly $|D|$, in agreement with Lemma~\ref {boundprime}.
We have partially implemented the search for a suitable curve: for each $4p
= u^2 - v^2 D$ we look for the first $j$-invariant such that for a random
point $P$ on an associated curve, $(p+1) P$ and $u P$ have the same
$X$-coordinate. This allows us to treat the curve and its quadratic twist
simultaneously. The largest occurring value of $v$ is~$5$. Altogether,
$487237$ curves need to be checked for the target cardinality. 

On an Athlon-64 2.2~GHz computer, this step takes roughly $18.5$ seconds. As
comparison, the third authors' complex analytic implementation takes 
$0.3$ seconds on the same machine. To speed up the multi-prime approach,
we incorporated some inert primes. Out of the $168$ primes less than $1000$,
there are $85$ primes that are inert in $\Or$. For many of them, the computation
of $H_D \bmod p$ is trivial. Together, these primes contribute $707$ bits and
we only need $288$ totally split primes, the largest one being $802597$. The
required $381073$ curve cardinalities are tested in $14.2$ seconds. 

One needs to be careful when drawing conclusions from only few examples, but
the difference between $14.2$ and $0.3$ seconds suggests that the implicit
constants in the $O$-symbol are worse for the CRT-approach.

\subsection{Class invariants}

For many applications, we are mostly interested in a generating polynomial
for the ring class field $K_\Or$. As the Hilbert class polynomial has very
large coefficients, it is then better to use `smaller functions' than the
$j$-function to save a constant factor in the size of the polynomials. We 
refer to \cite{Sc02,St00} for the theory of such {\it class invariants\/}.

There are theoretical obstructions to incorporating class invariants into
Algorithm~\ref {alg:crt}.
Indeed, if a modular function $f$ has the property that there
are class invariants $f(\tau_1)$ and $f(\tau_2)$ with different 
minimal polynomials, we cannot use the CRT-approach.
This phenomenon occurs for instance for the double eta quotients
described in \cite{EnSc04}. For the discriminant $D$ in 
Section~\ref {ssec:ex crt}, we can use the double eta quotient of level 
$3 \cdot 109$ to improve the $0.3$ seconds of the complex analytic approach.
For CRT, we need to consider less favourable class invariants.

\subsubsection*{Acknowledgement.}
We thank Dan Bernstein, Fran\c{c}ois Morain and Larry Washington for
helpful discussions.


\end{document}